\documentclass[letterpaper,article,oneside,11pt]{memoir}

\usepackage{amsmath}
\usepackage{amssymb}
\usepackage{amsthm}
\usepackage{mathrsfs}
\usepackage{lmodern}
\usepackage[T1]{fontenc}
\usepackage[utf8]{inputenc}
\usepackage{enumitem}
\usepackage{xcolor}

\settrims{0pt}{0pt}
\setheadfoot{0.5in}{0.5in}
\setulmarginsandblock{1.0in}{1.0in}{*}
\setlrmarginsandblock{1.0in}{1.0in}{*}
\checkandfixthelayout

\usepackage[backend=biber,style=alphabetic,giveninits,url=false,doi=false,maxbibnames=50,maxcitenames=50,maxnames=50]{biblatex}
\defbibheading{bibliography}[Bibliography]{\section*{#1}\markboth{#1}{#1}}

\counterwithout{equation}{section}
\counterwithout{section}{chapter}

\renewcommand{\maketitle}{
 {\centering\textbf{\thetitle}\par
  \vspace{0.5em}
  \centering{\theauthor}\par 
  \vspace{2em}
 }
}

\setsecheadstyle{\bfseries}

\setsecnumdepth{subsubsection}
\setsecnumformat{\csname the#1\endcsname. }
\setsubsecheadstyle{\normalfont\bfseries}
\setaftersubsecskip{-0.5em}
\setsubsubsecheadstyle{\normalfont\bfseries}
\setaftersubsubsecskip{-0.5em}

\OnehalfSpacing
\newtheorem{theorem}[equation]{Theorem}

\newtheorem*{claim}{\textit{Claim}}

\newtheorem{corollary}[equation]{Corollary}
\newtheorem{lemma}[equation]{Lemma}

\theoremstyle{definition}

\newtheorem{question}[equation]{Question}

\newcommand{\metric}{\mathsf{d}}
\newcommand{\nbar}{|\!|}
\newcommand{\intd}{\,\mathrm{d}}
\newcommand{\haar}{\mathsf{m}}
\newcommand{\Haar}{\mathsf{M}}

\newcommand{\sphere}{\mathsf{S}}
\newcommand{\bigo}{\mathsf{O}}
\newcommand{\N}{\mathbb{N}}
\newcommand{\Z}{\mathbb{Z}}
\newcommand{\R}{\mathbb{R}}
\newcommand{\SL}{\mathsf{SL}}
\DeclareMathOperator{\symdiff}{\triangle}

\newcommand{\define}[1]{\textbf{#1}}

\newcommand{\ceil}[1]{\lceil {#1} \rceil}
\renewcommand{\intd}{\,\mathsf{d}}
\newcommand{\hol}{\mathsf{h}}

\renewcommand{\nbar}[1][]{|\!|_\mathsf{#1}}

\newcommand{\geod}{\mathsf{g}}
\newcommand{\comp}{\mathsf{r}}

\newcommand{\Cesaro}{Ces\`{a}ro}
\newcommand{\Ruhr}{R\"{u}hr}
\newcommand{\Poincare}{Poincar\`{e}}

\title{Uniform distribution of saddle connection lengths}
\author{Jon Chaika \and Donald Robertson with an appendix by Daniel El-Baz \and Bingrong Huang}

\begin{document}

\maketitle

\begin{abstract}
For any $\SL(2,\R)$ invariant and ergodic probability measure on any stratum of flat surfaces, almost every flat surface has the property that its non-decreasing sequence of saddle connection lengths is uniformly distributed mod one.
\end{abstract}

\section{Introduction}

By a \define{flat surface} we mean a pair $(X,\omega)$ where $X$ is a closed, compact Riemann surface of genus $g \ge 2$ and $\omega$ is a non-zero holomorphic one form on $X$.
Every non-zero holomorphic one form on $X$ has $2g-2$ zeros where $g$ is the genus of $X$.
Fixing the genus of $X$ the set $\{ (X,\omega) : \omega \ne 0 \}$ is partitioned into \define{strata} according to the possible partitions of $2g-2$.
Each connected component of a stratum carries a natural volume measure that is invariant under the action of $\SL(2,\R)$ on pairs $(X,\omega)$.
The $\SL(2,\R)$ action preserves the area of $(X,\omega)$.
Masur~\cite{MR644018} and Veech~\cite{MR644019} showed independently that the natural volume on a connected component induces on its subset $\mathcal{H}$ of unit area flat surfaces a probability measure $\Haar$ that is ergodic for the $\SL(2,\R)$ actions. 

Fix a flat surface $(X,\omega)$.
Every non-zero holomorphic one form gives, away from its finite set $\Sigma$ of zeroes, an atlas of charts to $\R^2$ whose transition maps are translations in $\R^2$.
For every $0 \le \theta < \pi$ we can use such an atlas to induce a foliation on $X \setminus \Sigma$ given locally by the lines in $\R^2$ making an angle of $\theta$ with the horiontal axis.
We can also use such an atlas to measure lengths of curves on $X$.
In particular we can calculate the lengths of the leaves in the above foliations.
A \define{saddle connection} of $(X,\omega)$ is any leaf of any of these foliations that starts and ends at a point of $\Sigma$.
The \define{holonomy} of a saddle connection $v$ is the vector
\[
\hol(v) = \left( \int\limits_v \mathrm{Re}(\omega), \int\limits_v \mathrm{Im}(\omega) \right)
\]
in $\R^2$.

In this paper we are interested in the uniform distribution of saddle connection lengths.
Recall that a sequence $n \mapsto x_n$ in $\R$ is \define{uniformly distributed} mod 1 when
\[
\lim_{N \to \infty} \frac{ | \{ 1 \le n \le N : a \le x_n \bmod 1 < b \} | }{N} = b-a
\]
for all intervals $[a,b) \subset [0,1)$.
The well-known Weyl criterion is that uniform distribution of $n \mapsto x_n$ is equivalent to
\[
\lim_{N \to \infty} \frac{1}{N} \sum_{n=1}^N \exp(2 \pi i p x_n) = 0
\]
for all $p \in \Z \setminus \{0\}$ and one can use it to prove, for example, that $n \mapsto n \alpha$ is uniformly distributed for all irrational $\alpha$.

Returning to flat surfaces (and suppressing $X$ from our notation) we write $\Lambda(\omega)$ for the set of saddle connections of $\omega$.
For each $R > 0$ write $\Lambda(\omega;R)$ for the set of saddle connections whose holonomy vectors have length at most $R$.
Given also $S > 0$ write $\Lambda(\omega;R,S)$ for the set of saddle connections whose holonomy vectors lie within the ellipse $(R y)^2 + (S x)^2 = (RS)^2$.

It follows from minor modifications to work of Vorobets~\cite[Theorem~1.9]{MR2180238} that the projection of $\Lambda(\omega;R)$ to the unit circle along rays through the origin is uniformly distributed as $R \to \infty$ for almost every $\omega$.
Specifically, by \cite[Proposition~4.5]{MR2180238} it suffices to show that the set of saddle connection holonomies satisfy axioms (0), (A), (B), (C) and (E) therein.
That saddle connections doe satisfy these axioms follows from the definitions and results in \cite{MR1827113} just as it does for the sets Vorobets considers.

Our main result is that lengths of saddle connection holonomy vectors are also almost always uniformly distributed mod 1.
To make sense of this we fix, for each flat surface $\omega$, an enumeration $n \mapsto v_n$ of $\Lambda(\omega)$ such that $n \mapsto \nbar \hol(v_n) \nbar$ is non-decreasing.

\begin{theorem}
\label{thm:mainTheorem}
For any $\SL(2,\R)$ invariant and ergodic probability measure $\haar$ on $\mathcal{H}$ and $\haar$ almost every flat surface $\omega$ the sequence $n \mapsto \nbar \hol(v_n) \nbar$ is uniformly distributed.
\end{theorem}

Masur~\cite{MR955824,MR1053805} proved that for every surface $\omega$ there exists $c_1$, $c_2$ (depending on $\omega$) so that for all large $R$ we have
\begin{equation}
\label{eq:quad bounds}
c_1 R^2 \le |\Lambda(\omega;R)| < c_2 R^2
\end{equation}
(see also Vorobets \cite{arXiv:math/0307249}).
Veech~\cite{MR1670061} introduced what is now called the \define{Siegel-Veech} transform to show that there exists $c$ depending only on $\nu$ and $R_n \to \infty$ so that
\[
\lim_{i \to \infty} \frac{\Lambda(\omega;R_i)}{R_i^2}=c
\]
for $\Haar$ almost every $\omega$.
He also showed the limit exists for all surfaces in certain suborbifolds of $\mathcal{H}$. 
Building on Veech's approach, Eskin-Masur \cite{MR1827113} proved that for every $\SL(2,\R)$ invariant and $\SL(2,\R)$ ergodic probability measure $\haar$ there exists $c$ so that we have 
\begin{equation}\label{eq:a.e conv}
\underset{R \to \infty}{\lim}\, \frac{\Lambda(\omega;R)}{R^2}=c
\end{equation}
for $\haar$ almost every $\omega$.
Dozier \cite{arXiv:1705.10847} proved analogues of these results for saddle connections with holonomies in certain sectors of directions. 

The best known results for asymptotic counting that hold on all surfaces are Eskin, Mirzakhani and Mohammadi's result~\cite[Theorem 2.12]{MR3418528} that every surface has an asymptotic growth of saddle connections on average, and Dozier's result~\cite{arXiv:1701.00175} building on techniques in \cite{MR2787598} that the constants $c_1,c_2$ in \eqref{eq:quad bounds} can be chosen to depend only on the connected component of the stratum in question.
The best almost everywhere counting result is the following theorem of Nevo, \Ruhr{} and Weiss.

\begin{theorem}[{\cite[Theorem~1.1]{arXiv:1708.06263}}]
\label{thm:nrw}
For every $\SL(2,\R)$ invariant and ergodic probability measure $\haar$ on $\mathcal{H}$ there are constants $c,\kappa > 0$ such that $\haar$ almost every flat surface $\omega$ satisfies
\begin{equation}
\label{eqn:nrw}
|\Lambda(\omega; R)| = c R^2 + \bigo_\omega(R^{2(1 - \kappa)})
\end{equation}
for all $R > 0$.
\end{theorem}

The gaps of saddle connection directions has also been studied by Athreya, the first named author, Lelievre, Uyanik and Work \cite{MR3000496}, \cite{MR3330337} and \cite{MR3567252}.
Recently Athreya, Cheung and Masur~\cite{arXiv:1711.08537} have begun investigating the $\mathrm{L}^2$ properties of the Siegel-Veech transform. 

Most of these papers to a greater or lesser extent adhere to the philosophy of renormalization dynamics, using the $\SL(2,\mathbb{R})$ action and especially the geodesic flow
\[
\geod^t = \begin{bmatrix} e^{t} & 0\\ 0 &e^{-t} \end{bmatrix}
\]
to translate questions about long saddle connections to moderate saddle connections on different surfaces.
Our approach is a little different: we still use the $\SL(2,\R)$ action, but only as a means to perturb.
Namely, we do not apply the geodesic flow $\geod^t$ long enough to make a saddle connection $\bigo(1)$ in length.
Instead, writing also
\[
\comp^\theta = \begin{bmatrix} \cos \theta & -\sin \theta \\ \sin \theta & \cos \theta \end{bmatrix}
\]
we prove that for almost every $\omega$ and for almost every $\theta, t$ the lengths of saddle connections on $\geod^t \comp^{\theta} \omega$ are uniformly distributed mod 1.
However, we emphasize that we do use Theorem~\ref{thm:nrw}, which does use renormalization dynamics. 

We remark that even if the conclusion of \cite[Theorem~1.1]{arXiv:1708.06263} were known for \textit{every} flat surface, our methods would not give uniform distribution of saddle connection lengths for every flat surface.
Therefore, the following question is in general open.

\begin{question}
Is the sequence $n \mapsto \nbar \hol(v_n) \nbar$ uniformly distributed for every flat surface?
\end{question}

However, in the presence a very strong error term for the saddle connection counting function $|\Lambda(\omega;R)|$ one can obtain uniform distribution of saddle connection lengths everywhere with a much simpler proof.
The appendix by Daniel El-Baz and Bingrong Huang proves using Huxley and Nowak's bound~\cite{MR1397317} on the counting function for primitive lattice points that the conclusion of Theorem~\ref{thm:mainTheorem} holds for every torus.

The research of J.\ Chaika was supported in part by NSF grants DMS-135500 and DMS-1452762, the Sloan foundation, a \Poincare{} chair, and a Warnock chair.
J.\ Chaika thanks Alex Eskin for encouraging him to pursue this question.
The research of D.\ Robertson was supported by NSF grant DMS-1703597.

\section{Proof of Theorem~\ref{thm:mainTheorem}}
\label{sec:saddleLengthNRW}

In this section we prove Theorem~\ref{thm:mainTheorem}.
Let $\mathcal{H}$ be the set of unit-area flat surfaces in the connected component of a stratum and write $\haar$ for the Masur-Veech probability measure $\haar$ on $\mathcal{H}$.
Let $c > 0$ and $\kappa > 0$ be as in Theorem~\ref{thm:nrw}.
Our proof of Theorem~\ref{thm:mainTheorem} has the following four ingredients.
\begin{enumerate}
\item
The Weyl criterion.
\item
An exhaustion by compact sets.
\item
A relation with orbit averages.
\item
A linear approximation to handle exponential sums.
\end{enumerate}
These steps are covered in the following four subsections.
In Subsection~\ref{subsec:finale} we combine the steps to prove Theorem~\ref{thm:mainTheorem}.

\subsection{The Weyl criterion}

By the Weyl criterion for uniform distrubtion the following result implies Theorem~\ref{thm:mainTheorem}.
Recall that $n \mapsto \hol(v_n)$ is an enumeration of $\Lambda(\omega)$ such that $n \mapsto \nbar \hol(v_n) \nbar$ is non-decreasing.

\begin{theorem}
\label{thm:weylForLengths}
For $\haar$ almost every $\omega$ we have
\begin{equation}
\label{eqn:weylForLengths}
\lim_{N \to \infty} \frac{1}{N} \sum_{n=1}^N \exp(2 \pi i p \nbar \hol(v_n) \nbar) = 0
\end{equation}
for every $p \in \N$.
\end{theorem}

Fix $p \in \N$ and write $\chi_p(x) = \exp(2 \pi i p x)$ for any $x \in \R$.
To prove that \eqref{eqn:weylForLengths} holds almost surely we use the following theorem.

\begin{lemma}
\label{lem:tauberian}
Fix a flat surface $\omega$.
If there is a sequence $\tau_1 > \tau_2 > \tau_3 > \cdots \to 1$ such that
\begin{equation}
\label{eqn:tauberianForLengths}
\lim_{J \to \infty} \frac{1}{\ceil{\tau_i^J}} \sum_{n=1}^{\ceil{\tau_i^J}} \chi_p (\nbar \hol(v_n) \nbar) = 0
\end{equation}
for all $i \in \N$ then \eqref{eqn:weylForLengths} holds.
\end{lemma}
\begin{proof}
This is a fact about \Cesaro{} convergence.
Fix a sequence $\tau_1 > \tau_2 > \tau_3 > \cdots \to 1$.
Suppose that a bounded sequence $n \mapsto b_n$ has the property that
\[
\lim_{J \to \infty} \frac{1}{\ceil{\tau_i^J}} \sum_{n=1}^{\ceil{\tau_i^J}} b_n = \beta
\]
for all $i \in \N$.
We verify that $n \mapsto b_n$ \Cesaro{} converges to $\beta$.
Indeed, fix $\epsilon > 0$ and $i \in \N$ such that
\begin{equation}
\label{eqn:tauberianChoice}
1 < \tau_i < 1 + \frac{\epsilon}{4 \nbar b \nbar[\infty]}
\end{equation}
and choose $J \in \mathbb{N}$ so large that
\[
\max \left\{ \left| \frac{1}{\ceil{\tau_i^j}} \sum_{n=1}^{\ceil{\tau_i^j}} b_n - \beta \right|, \frac{2 \nbar b \nbar[\infty]}{\ceil{\tau_i^j}} \right\}
<
\frac{\epsilon}{3}
\]
holds for all $j \ge J$.
Put $N = \ceil{\tau_i^J}$.
Given $M > N$ write $M = \ceil{\tau_i^j} + k$ with $k \in \mathbb{N}$ and $j$ maximal.
Now
\[
\tau_i^j + k
\le
\ceil{\tau_i^j} + k
<
\tau_i^{j+1} + 1
<
\left( 1 + \frac{\epsilon}{4 \nbar b \nbar[\infty]} \right) \tau_i^j + 1
\]
giving $4k \nbar b \nbar[\infty] \le \epsilon \tau_i^j + 4 \nbar b \nbar[\infty]$.
So
\[
\left| \frac{1}{\ceil{\tau_i^j}} \sum_{n=1}^{\ceil{\tau_i^j}} b_n - \frac{1}{\ceil{\tau_i^j} + k} \sum_{n=1}^{\ceil{\tau_i^j} + k} b_n \right|
\le
\frac{2k \nbar b \nbar[\infty]}{\ceil{\tau_i^j} + k}
\le
\frac{\epsilon}{2} \frac{\tau_i^j}{\ceil{\tau_i^j} + k} + \frac{2 \nbar b \nbar[\infty]}{\ceil{\tau_i^j} + k}
\le
\epsilon
\]
as desired.
\end{proof}

We conclude this subsection by replacing \eqref{eqn:tauberianForLengths} with an average over all saddle connections whose holonomy vectors have length at most $\sqrt{\ceil{\tau^J}/c}$.

\begin{lemma}
\label{lem:orderToDiscs}
If
\begin{equation}
\label{eqn:orderToDiscs}
\lim_{J \to \infty} \frac{c}{\ceil{\tau^J}} \sum_{v \in \Lambda \left( \omega;\sqrt{\ceil{\tau^J}/c} \right) } \chi_p(\nbar \hol(v) \nbar) = 0
\end{equation}
for $\haar$ almost every $\omega$ then
\[
\lim_{J \to \infty} \frac{1}{\ceil{\tau^J}} \sum_{n=1}^{\ceil{\tau^J}} \chi_p(\nbar \hol(v_n) \nbar) = 0
\]
for $\haar$ almost every $\omega$.
\end{lemma}
\begin{proof}
For almost every $\omega$ there is from Theorem~\ref{thm:nrw} a constant $C_\omega > 0$ with
\[
\left|
\sum_{n=1}^{\ceil{\tau^J}} \chi_p( \hol \nbar v_n \nbar)
-
\sum_{v \in \Lambda \left( \omega;\sqrt{\ceil{\tau^J}/c} \right) } \chi_p( \hol(\nbar v \nbar))
\right|
\le
\frac{C_\omega}{c^{1 - \kappa}} \ceil{\tau^J}^{1 - \kappa}
\]
for all $J \in \N$.
Dividing by $\ceil{\tau^J}$ gives the desired result.
\end{proof}

\subsection{An exhaustion by compact sets}

To relieve notation slightly we now consider the expression
\begin{equation}
\label{eqn:radiusAverage}
\frac{1}{R^2} \sum_{v \in \Lambda(\omega;R)} \chi_p( \nbar \hol(v) \nbar)
\end{equation}
for $R > 0$.
Our goal is to estimate the size of this average.
We do so on compact sets that exhaust a full-measure subset of $\mathcal{H}$.
To define our compact sets write $\ell(\omega)$ for the length of the shortest saddle connection of $\omega$ and
\[
\xi(\omega) = \sup \left\{ \frac{\Big| |\Lambda(\omega;R)| - c R^2 \Big|}{R^{2(1-\kappa)}} : R > 0 \right\}
\]
for the $\haar$ almost-surely defined optimal constant in \eqref{eqn:nrw}.
Fix $r > 0$ small and $\Xi > 0$ large.
Let $K$ be a compact subset of $\ell^{-1}[r,\infty)$ on which both
\[
\haar(\ell^{-1}[r,\infty) \setminus K) < \frac{1}{r}
\]
and
\begin{equation}
\label{eqn:boundOnK}
\sup \{ \xi(\omega) : \omega \in K \} \le \Xi
\end{equation}
hold.
We can find such a compact set because $\xi$ is measurable.
As $r \to 0$ and $\Xi \to \infty$ our compact set $K$ exhausts $\mathcal{H}$ up to a set of $\haar$ measure zero.

To describe the behavior of \eqref{eqn:radiusAverage} on $K$ we introduce some constants.
Fix $0 < \eta < 1$ and choose $\rho$ such that
\[
1 > \rho > \max \left\{ \frac{1+\eta}{2} ,1 - 2\kappa \right\}
\]
with $\kappa$ coming from Theorem~\ref{thm:nrw}.
There is $\delta > 0$ such that
\begin{equation}
\label{eqn:delta}
\left| \frac{\exp(2 \pi i s) - 1}{2\pi i s} - 1 \right| < \epsilon
\end{equation}
whenever $|s| < \delta$.
We only consider below values of $R$ large enough that
\begin{equation}
\label{eqn:radius}
R^{1 + \eta}
\ge
\max \left\{ \frac{5\epsilon p^2}{42 \delta^2}, \frac{\epsilon}{42} \right\}
\end{equation}
and fix the relationship
\begin{equation}
\label{eqn:geodTime}
S = \sqrt{\frac{\epsilon}{42 R^{1+\eta}}}
\end{equation}
between $S$ and $R$.
Note that \eqref{eqn:radius} implies $S \le 1$.
Although $S$ and $R$ are related we think of $S$ as the small amount of time we will geodesic flow for, and $R$ as a large bound on the lengths of our saddle connections.

\subsection{A relation with orbit averages}

The main result of this section is a relation between \eqref{eqn:radiusAverage} on $K$ and its $\SL(2,\R)$ average over a small disc.

\begin{theorem}
\label{thm:orbitSmear}
We have
\begin{equation}
\label{eqn:orbitSmear}
\begin{aligned}
&
\haar \left(
\left\{
\omega \in K : \left| \frac{1}{R^2} \sum_{v \in \Lambda(\omega;R)} \chi_p(\nbar \hol(v) \nbar) \right|^2
\ge
\frac{1}{R^\sigma}
\right\} \right)^2
\\
={}
&
\bigo \Bigg(
R^\sigma
\int\limits_K
\frac{1}{2\pi} \int\limits_0^{2\pi}
\frac{1}{S} \int\limits_0^S
\left|
\frac{1}{R^2} \sum_{v \in \Lambda(\geod^t \comp^\theta \omega;R)} \chi_p(\nbar \hol(v) \nbar)
\right|^2
\intd t
\intd \theta
\intd \haar(\omega)
\Bigg)
\end{aligned}
\end{equation}
for all $R$ satisfying \eqref{eqn:radius} and all $\sigma > 0$.
\end{theorem}
\begin{proof}
Let
\[
N = \left\{ \omega \in K : \left| \frac{1}{R^2} \sum_{v \in \Lambda(\omega;R)} \chi_p(\nbar \hol(v) \nbar) \right|^2 \ge \frac{1}{R^{\sigma}} \right\}
\]
be the set whose measure we wish to bound.
Write $\mu$ for Haar measure on $\SL(2,\R)$, which is left and right invariant because $\SL(2,\R)$ is unimodular.
For each $s > 0$ write
\[
D(s) = \{ [\begin{smallmatrix} a & b \\ c & d \end{smallmatrix}] \in \SL(2,\R) : \metric(i, \tfrac{ai+b}{ci+d}) < s \}
\]
where $\metric$ is the hyperbolic distance on the upper half-plane determined by the metric of constant curvature $-4$.
In this metric the area of a hyperbolic disc of radius $r$ is $\pi (\sinh(\frac{r}{2}))^2$ so whenever $s < 1$ we have
\begin{equation}
\label{eqn:doublingConstant}
1 \le \frac{\mu(D(2s))}{\mu(D(s))} = (2 \cosh(\tfrac{s}{2}) )^2 \le 9
\end{equation}
using the hyperbolic double angle formula.

The defining property of $N$ gives
\begin{align*}
&
\frac{1}{R^{\sigma}}
\int\limits_N
\frac{1}{\mu(D(S))}
\int\limits_{D(S)} 1_N(g \omega) \intd \mu(g) \intd \haar(\omega)
\\
\le
&
\int\limits_K
\frac{1}{\mu(D(S))} \int\limits_{D(S)}
\left| \frac{1}{R^2} \sum_{v \in \Lambda(g\omega;R)} \chi_p(\nbar \hol(v) \nbar) \right|^2
\intd \mu(g) \intd \haar(\omega)
\\
=
&
\int\limits_K
\frac{1}{(\sinh(\frac{S}{2}))^2} \int\limits_0^{2\pi} \int\limits_0^{2\pi} \int\limits_0^S
\sinh(t) \left| \frac{1}{R^2} \sum_{v \in \Lambda(\comp^{\theta_1} \geod^t \comp^{\theta_2}\omega;R)} \chi_p(\nbar \hol(v) \nbar) \right|^2
\intd t \intd \theta_1 \intd \theta_2 \intd \haar(\omega)
\end{align*}
after using the Cartan integral formula~\cite[Proposition~5.28]{MR1880691} with $\mu$ normalized appropriately.
We can eliminate the integral over $\theta_1$ because the rotation $\comp^{\theta_1}$ doesn't change the sum.
Together with some simple estimates we arrive at
\[
\frac{1}{R^{\sigma}}
\int\limits_N
\frac{1}{\mu(D(S))}
\int\limits_{D(S)} 1_N(g \omega) \intd \mu(g) \intd \haar(\omega)
\le
C
\int\limits_K
\frac{1}{2\pi} \int\limits_0^{2\pi}
\frac{1}{S} \int\limits_0^S 
\left| \frac{1}{R^2} \sum_{v \in \Lambda(\geod^t \comp^\theta \omega;R)} \chi_p(\nbar \hol(v) \nbar) \right|^2
\intd t \intd \theta \intd \haar(\omega)
\]
where $C$ is some absolute positive constant because $\sinh(S) \sim S$.
It therefore suffices by \eqref{eqn:doublingConstant} to prove
\begin{equation}
\label{eqn:smearBound}
\int\limits_N \frac{1}{\mu(D(S/4))}
\int\limits_{D(S)}
1_N(g\omega)
\intd \mu(g)
\intd \haar(\omega)
\ge
\frac{\haar(N)^2}{4}
\end{equation}
holds.

Given $E \subset \mathcal{H}$ and $\omega \in \mathcal{H}$ write $E \omega^{-1} = \{ g \in \SL(2,\R) : g \omega \in E \}$.
Put
\[
h(\omega) = \frac{\mu(N \omega^{-1} \cap D(S/4))}{\mu(D(S/4))}
\]
and note that
\begin{equation}
\label{eqn:measureBadSet}
\begin{aligned}
\haar(N)
&
=
\frac{1}{\mu(D(S/4))} \int\limits_{D(S/4)} \int\limits_{\mathcal{H}} 1_{g^{-1} N}(\omega) \intd \haar(\omega) \intd \mu(g)
\\
&
=
\int\limits_{\mathcal{H}} \frac{1}{\mu(D(S/4))} \int\limits_{D(S/4)} 1_N(g \omega) \intd \mu(g) \intd \haar(\omega)
=
\int\limits_\mathcal{H} h \intd \haar
\end{aligned}
\end{equation}
by invariance of $\haar$ and Fubini.
Write
\[
U_q = \bigcup_{h(\omega) \ge q} D(S/2) \cdot \omega
\]
for any $0 < q < 1$.
As in \eqref{eqn:measureBadSet} we have
\[
\haar(N \cap U_q)
=
\int\limits_\mathcal{H} \frac{1}{\mu(D(S/4))} \int\limits_{D(S/4)} 1_{N \cap U_q}(g \omega) \intd \mu(g) \intd \haar(\omega)
\]
for every $0 < q < 1$ and it follows for every $0 < q < 1$ that
\begin{equation}
\label{eqn:ballRelation}
\haar(N \cap U_q)
\ge
\int\limits_{h^{-1}[q,1]} \frac{\mu(N \omega^{-1} \cap U_q \omega^{-1} \cap D(S/4))}{\mu(D(S/4))} \intd \haar(\omega)
=
\int\limits_{h^{-1}[q,1]} h(\omega) \intd \haar(\omega)
\ge
q \haar(h^{-1}[q,1])
\end{equation}
because $U_q \omega^{-1} \supset D(S/2)$ whenever $h(\omega) \ge q$.

\begin{claim}
$\displaystyle \haar(N \cap U_q) \le \haar \left( \left\{ \omega \in N : \frac{\mu(N \omega^{-1} \cap D(S))}{\mu(D(S/4))} \ge q \right\} \right)$
\end{claim}
\begin{proof}
If $\omega \in N \cap U_q$ then there is $\eta$ with $h(\eta) \ge q$ and $\omega \in D(S/2) \cdot \eta$.
Fix $a \in D(S/2)$ with $a \cdot \eta = \omega$.
We get
\[
q
\le
\frac{ \mu((N \eta^{-1})a^{-1} \cap D(S/4)a^{-1}) }{ \mu(D(S/4)) }
=
\frac{ \mu(N \omega^{-1} \cap D(S/4)a^{-1}) }{ \mu(D(S/4)) }
\le
\frac{ \mu(N \omega^{-1} \cap D(S)) }{ \mu(D(S/4)) }
\]
since $\mu$ is invariant on both sides and $D(S/4)a^{-1} \subset D(S)$
\end{proof}

The claim, combined with \eqref{eqn:ballRelation}, gives
\begin{align*}
&
\int\limits_N \frac{1}{\mu(D(S/4))} \int\limits_{D(S)} 1_N(g \omega) \intd \mu(g) \intd \haar(\omega)
\\
={}
&
\int\limits_0^\infty \haar \left( \left\{ \omega \in N : \frac{\mu( N \omega^{-1} \cap D(S))}{\mu(D(S/4))} \ge q \right\} \right) \intd q
\\
\ge{}
&
\int\limits_\frac{\haar(N)}{2}^1 q \haar(h^{-1}[q,1]) \intd q
\\
\ge{}
&
\frac{\haar(N)}{2} \int\limits_\frac{\haar(N)}{2}^1 \haar(h^{-1}[q,1]) \intd q
\ge
\frac{\haar(N)^2}{4}
\end{align*}
where the last inequality follows from
\[
\int\limits_\mathcal{H} h \intd \haar
=
\int\limits_0^\frac{\haar(N)}{2} \haar(h^{-1}[q,1]) \intd q + \int\limits_\frac{\haar(N)}{2}^1 \haar(h^{-1}[q,1]) \intd q
\le
\frac{\haar(N)}{2} + \int\limits_\frac{\haar(N)}{2}^1 \haar(h^{-1}[q,1]) \intd q
\]
and \eqref{eqn:measureBadSet}.
This establishes \eqref{eqn:smearBound} as desired.
\end{proof}

The following lemma, another application of Theorem~\ref{thm:nrw}, will allow us to move the action of $\SL(2,\R)$ from $\omega$ to the holonomy vectors in the summation.

\begin{lemma}
\label{lem:ellipseEstimate}
For every $\omega \in K$ we have
\[
\left|
\frac{1}{R^2}
\sum_{v \in \Lambda(\geod^t \omega;R)}
\chi_p(\nbar \hol(v) \nbar)
-
\frac{1}{R^2}
\sum_{v \in \Lambda(\omega;R)}
\chi_p(\nbar \geod^t \hol(v) \nbar)
\right|
\le
\frac{8 \Xi}{ R^{2\kappa}} + 8c \sqrt{\frac{\epsilon}{42 R^{1 + \eta}}}
\]
all $0 \le t \le S$.
\end{lemma}
\begin{proof}
Fix $R > 0$ and $\omega \in K$ and $t \ge 0$.
First note that $v$ belongs to $\Lambda(\geod^t \omega;R)$ if and only if $\geod^{-t} \hol(v)$ belongs to the ellipse of width $2e^{-t} R$ and height $2e^t R$.
Thus
\[
\sum_{v \in \Lambda(\geod^t \omega;R)}
\chi_p(\nbar \hol(v) \nbar)
=
\sum_{v \in \Lambda(\omega;e^{-t} R, e^t R)}
\chi_p(\nbar \geod^t \hol(v) \nbar)
\]
holds for all $R > 0$.
Next we estimate $|\Lambda(\omega; e^t R,e^{-t} R) \symdiff \Lambda(\omega;R)|$.
For every $\omega$ we have
\begin{align*}
&
|\Lambda(\omega; e^t R,e^{-t} R) \symdiff \Lambda(\omega;R)|
\\
\le{}
&
|\Lambda(\omega;e^t R)| - |\Lambda(\omega;e^{-t} R)|
\\
\le{}
&
(e^t R)^{2(1 - \kappa)} \Xi + c (e^t R)^2 + (e^{-t} R)^{2(1 - \kappa)} \Xi - c (e^{-t} R)^2 
\\
\le{}
&
2 \Xi \cosh(2t(1 - \kappa)) R^{2(1 - \kappa)} + 2 c \sinh(2t) R^2
\end{align*}
for all $t \ge 0$ using \eqref{eqn:boundOnK}, as we may since $\omega \in K$.
It follows from the above that
\[
\left|
\sum_{v \in \Lambda(\geod^t \omega;R)}
\chi_p(\nbar \hol(v) \nbar)
-
\sum_{v \in \Lambda(\omega;R)}
\chi_p(\nbar \geod^t \hol(v) \nbar)
\right|
\le
8 \Xi R^{2(1 - \kappa)} + 8 c S R^2
\]
for all $0 \le t \le S$.
The conclusion follows from \eqref{eqn:geodTime}.
\end{proof}

\begin{corollary}
\label{cor:ellipseEstimate}
For every $\omega \in K$ we have
\begin{equation}
\label{eqn:movingAction}
\left|
\frac{1}{R^2}
\sum_{v \in \Lambda(\geod^t \comp^\theta \omega;R)}
\chi_p(\nbar \hol(v) \nbar)
-
\frac{1}{R^2}
\sum_{v \in \Lambda(\omega;R)}
\chi_p(\nbar \geod^t \comp^\theta \hol(v) \nbar)
\right|
\le
\frac{8 \Xi}{ R^{2\kappa}} + 8c \sqrt{\frac{\epsilon}{42 R^{1 + \eta}}}
\end{equation}
for all $0 \le \theta < 2\pi$ and all $0 \le t \le S$.
\end{corollary}
\begin{proof}
Apply the lemma with $\comp^\theta \omega$ in place of $\omega$.
Then note that when $t = 0$ the two sums are equal.
\end{proof}

\subsection{A linear approximation}

In this section we estimate the right-hand side of \eqref{eqn:orbitSmear}, proving the following result.

\begin{theorem}
\label{thm:mainEstimate}
There is $0 < \gamma < 1$ such that
\begin{equation}
\label{eqn:wantBound}
\int\limits_K
\frac{1}{2 \pi} \int\limits_0^{2\pi}
\frac{1}{S} \int\limits_0^S
\frac{1}{R^4}
\sum_{v \in \Lambda(\omega;R)}
\sum_{w \in \Lambda(\omega;R)}
\chi_p(\nbar \geod^t \comp^\theta \hol( v ) \nbar) \, \overline{\chi_p(\nbar \geod^t \comp^\theta \hol( w )\nbar)}
\intd t \intd \theta \intd \haar(\omega)
=
\bigo \left( \frac{1}{R^\gamma} \right)
\end{equation}
for almost all $\omega \in K$.
\end{theorem}

For the proof of Theorem~\ref{thm:mainEstimate} we will need several lemmas.
Write
\[
\alpha(u) = \frac{u_1^2 - u_2^2}{\nbar u \nbar}
\qquad
\qquad
\beta(u) = \frac{2 u_1 u_2}{\nbar u \nbar}
\]
for any $u = (v_1,v_2)$ in $\R^2$.
Certainly $|\alpha(u)| \le \nbar u \nbar$ and $|\beta(u)| \le \nbar u \nbar$ for all $u \in \R^2$.

\begin{lemma}
\label{lem:flowLengthEstimate}
We have
\begin{equation}
\label{eqn:flowLengthEstimate}
\Big| \nbar u \nbar + \alpha(u) t - \nbar \geod^t u \nbar \Big| < \frac{\epsilon}{R^\eta}
\end{equation}
whenever $0 \le t \le S$ and $u \in \R^2$ satisfies $\nbar u \nbar \le R$.
\end{lemma}
\begin{proof}
Fix $u \in \R^2$ with $\nbar u \nbar \le R$.
Writing $f_u(t) = \nbar \geod^t u \nbar$ whenever $0 \le t \le S$ note that
\begin{equation}
\label{eqn:flowLengthSize}
e^{-t} \nbar u \nbar \le f_u(t) \le e^t \nbar u \nbar
\end{equation}
and
\[
f_u''(t) = 2 f_u(t) - \frac{(e^{2t} u_1^2 - e^{-2t} u_2^2)^2}{f_u(t)^3} = f_u(t) + \frac{4 u_1^2 u_2^2}{f_u(t)^3}
\]
for all $t$.
Fix $\epsilon > 0$.
From \eqref{eqn:geodTime} and \eqref{eqn:flowLengthSize} we have
\[
\left| \frac{f''_u(t)}{2} t^2 \right|
\le
\frac{t^2}{2} ( e^t + 4 e^{3t} ) \nbar u \nbar
\le
42 t^2 \nbar u \nbar
\le
42 S^2 R
=
\frac{\epsilon}{R^\eta}
\]
for all $0 < t < S$.
Therefore \eqref{eqn:flowLengthEstimate} follows from the Lagrange form of the remainder in Taylor's theorem.
\end{proof}

We now focus on
\begin{equation}
\label{eqn:afterEllipse}
\frac{1}{2 \pi} \int\limits_0^{2\pi}
\frac{1}{S} \int\limits_0^S
\chi_p(\nbar \comp^\theta \hol(v) \nbar + t \alpha(\comp^\theta \hol(v)))
\,
\overline{\chi_p (\nbar \comp^\theta \hol(w) \nbar + t \alpha(\comp^\theta \hol(w)))}
\intd t \intd \theta
\end{equation}
for fixed $v,w \in \Lambda(\omega;R)$.
Note that
\begin{equation}
\label{eqn:absoluteValue}
|\eqref{eqn:afterEllipse}|
=
\left| \frac{1}{2\pi} \int\limits_0^{2\pi} \frac{1}{S} \int\limits_0^S \chi_p(t \alpha(\comp^\theta \hol(v)))
\,
\overline{\chi_p(t \alpha(\comp^\theta \hol(w)))} \intd t \intd \theta \right|
\end{equation}
because $\nbar \cdot \nbar$ is $\comp$ invariant.

\begin{lemma}
\label{lem:radialFriends}
If $v,w \in \Lambda(\omega;R)$ satisfy
\begin{equation}
\label{eqn:radialFriends}
\Big| \nbar \hol(v) \nbar - \nbar \hol(w) \nbar \Big|
\ge
R^\rho
\end{equation}
then
\[
| \eqref{eqn:afterEllipse} |
\le
4 \cdot \frac{1 + \epsilon}{2 \pi R}
+
\frac{4}{\pi^2} \sqrt{\frac{42 R^\eta}{\epsilon}} \frac{1}{R^\rho} \left( \log R + \log \frac{\pi}{4} \right)
\]
holds.
\end{lemma}
\begin{proof}
In absolute value \eqref{eqn:afterEllipse} is equal to
\[
\left| \frac{1}{2\pi} \int\limits_0^{2\pi} \frac{\exp(2\pi i S p ( \alpha(\comp^\theta \hol(v)) - \alpha(\comp^\theta \hol(w)) )) - 1}{2\pi i S p ( \alpha(\comp^\theta \hol(v)) - \alpha(\comp^\theta \hol(w)) )} \intd \theta \right|
\]
by using \eqref{eqn:absoluteValue} then integrating over $t$.
First, by making a substitution, we can assume $\hol(w)$ is horizontal.
We have $\beta(\hol(w)) = 0$ and $\alpha(\hol(w)) = \nbar \hol(w) \nbar$.
So
\begin{align*}
&
\alpha(\comp^\theta \hol(v)) - \alpha(\comp^\theta \hol(w))
\\
={}
&
\Big( \alpha(\hol(v)) \cos(2\theta) - \beta(\hol(v)) \sin(2\theta) \Big)
-
\Big( \alpha(\hol(w)) \cos(2\theta) - \beta(\hol(w)) \sin(2\theta) \Big)
\\
={}
&
\Big( \alpha(\hol(v)) - \nbar \hol(w) \nbar \Big) \cos(2\theta) - \beta(\hol(v)) \sin(2\theta)
\\
={}
&
A(v,w) \sin(\phi - 2 \theta)
\end{align*}
where $A(v,w)^2 = \left( \alpha(\hol(v)) - \nbar \hol(w) \nbar \right)^2 + \beta(\hol(v))^2$ and $\phi$ is chosen appropriately using the angle addition formula.
We are therefore interested in
\[
\left|
\frac{1}{2\pi} \int\limits_0^{2\pi} \frac{\exp(2\pi i S p A(v,w) \sin(\phi - 2\theta) ) - 1}{2\pi i S p A(v,w) \sin(\phi - 2\theta)} \intd \theta
\right|
\]
and consider separately the integral over small intervals centered at zeros of $\sin(\phi - 2\theta)$ and what remains.
Precisely, if $I$ is an interval of radius $1/R$ centered at a zero of $\sin(\phi - 2\theta)$ then
\[
|S p A(v,w) \sin(\phi - 2\theta)|
\le
p \sqrt{\frac{5\epsilon}{42}} R^{1 - (1+\eta)/2} |\sin(\phi - 2\theta)|
\le
p \sqrt{\frac{5\epsilon}{42}} \frac{1}{R^{(1+\eta)/2}}
<
\delta
\]
for all $\theta \in I$ by \eqref{eqn:radius}, giving
\[
\left| \frac{\exp(2\pi i S p A(v,w) \sin(\phi - 2\theta) ) - 1}{2\pi i S p A(v,w) \sin(\phi - 2\theta)} - 1 \right| \le \epsilon
\]
for every $\theta \in I$ by definition of $\delta$ (see \eqref{eqn:delta}).
There being four such intervals, we can estimate
\begin{equation}
\label{eqn:mainEstimate}
\begin{aligned}
&
\left| \frac{1}{2\pi} \int\limits_0^{2\pi} \frac{\exp(2\pi i S p A(v,w) \sin(\phi - 2\theta) ) - 1}{2\pi i S p A(v,w) \sin(\phi - 2\theta)} \intd \theta \right|
\\
\le{}
&
4 \cdot \frac{1+ \epsilon}{2 \pi R}
+
\frac{4}{2\pi} \int\limits_{\frac{\phi}{2}+\frac{1}{R}}^{\frac{\phi}{2} + \frac{\pi}{2} - \frac{1}{R}} \frac{2}{2 \pi S A(v,w) |\sin(\phi - 2\theta)|} \intd \theta
\end{aligned}
\end{equation}
by trivially estimating the numerator on the complement of the four intervals.
So (after substituting away $\phi$) we wish to make
\begin{equation}
\label{eqn:estimate}
\frac{1}{\pi S A(v,w)} \frac{4}{\pi} \int\limits_{\frac{1}{R}}^{\frac{\pi}{4}} \frac{1}{\sin(2\theta)} \intd \theta
\end{equation}
small.
Now $\frac{4}{\pi} \theta \le \sin(2 \theta)$ on $[0,\frac{\pi}{4}]$ so
\[
\frac{4}{\pi} \int\limits_{\frac{1}{R}}^{\frac{\pi}{4}} \frac{1}{\sin(2\theta)} \intd \theta
\le
\int\limits_{\frac{1}{R}}^{\frac{\pi}{4}} \frac{1}{\theta} \intd \theta
=
\log \frac{\pi}{4} + \log R
\]
holds.
With \eqref{eqn:geodTime} the quantity \eqref{eqn:estimate} becomes
\[
\frac{4}{\pi^2} \sqrt{\frac{42 R^{1+\eta}}{\epsilon}} \frac{1}{A(v,w)} \left( \log R + \log \frac{\pi}{4} \right)
\]
so we wish to determine when $A(v,w)$ not too small.
So we estimate the size of $A(v,w)$ recalling that $\hol(w)$ is horizontal.
Let $\theta_{v,w}$ be the angle between $\hol(v)$ and $\hol(w)$ so that $\hol(v) = \nbar \hol(v) \nbar (\cos \theta_{v,w},\sin \theta_{v,w})$.
Then
\begin{align*}
A(v,w)^2
&
=
(\nbar \hol(v) \nbar \cos (2\theta_{v,w}) - \nbar \hol(w) \nbar )^2 + ( \nbar \hol(v) \nbar \sin (2 \theta_{v,w}) )^2
\\
&
=
\nbar \hol(v) \nbar^2 + \nbar \hol(w) \nbar^2 - 2 \nbar \hol(v) \nbar \nbar \hol(w) \nbar \cos(2 \theta_{v,w})
\\
&
\ge
(\nbar \hol(v) \nbar - \nbar \hol(w) \nbar)^2
\end{align*}
and \eqref{eqn:radialFriends} concludes the proof.
\end{proof}

\begin{lemma}
\label{lem:unfriendlyCount}
Fix $\omega \in K$.
We have
\[
\left| \left\{ u \in \Lambda(\omega;R) : \Big| \nbar \hol(v) \nbar - \nbar \hol(u) \nbar \Big| \le R^\rho \right\} \right|
\le
(4c + 8\Xi) R^{1 + \rho}
\]
for every $v \in \Lambda(\omega;R)$.
\end{lemma}
\begin{proof}
The cardinality of the set is the same as the number of saddle connections of $\omega$ in the annulus $\sphere^1 \times [\nbar \hol(v) \nbar - R^\rho,\nbar \hol(v) \nbar + R^\rho]$.
By Theorem~\ref{thm:nrw} this cardinality is at most
\begin{align*}
&
|c(\nbar \hol(v) \nbar + R^\rho)^2 - c(\nbar \hol(v) \nbar - R^\rho)^2 + \Xi(\nbar \hol(v) \nbar + R^\rho)^{2(1 - \kappa)} + \Xi(\nbar \hol(v) \nbar - R^\rho)^{2(1 - \kappa)}|
\\
\le{}
&
4 c \nbar \hol(v) \nbar R^\rho
+
2 \Xi ( \nbar \hol(v) \nbar + R^\rho )^{2(1 - \kappa)}
\\
\le{}
&
4 c R^{1+\rho} + 8 \Xi R^{2(1 - \kappa)}
\end{align*}
by choice of $\rho$.
\end{proof}

We can now give the proof of Theorem~\ref{thm:mainEstimate}.

\begin{proof}
[Proof of Theorem~\ref{thm:mainEstimate}]
It suffices to prove that
\[
\int\limits_K
\frac{1}{R^4}
\sum_{v \in \Lambda(\omega;R)}
\sum_{w \in \Lambda(\omega;R)}
\frac{1}{2\pi} \int\limits_0^{2\pi} \frac{1}{S} \int\limits_0^S
\chi_p(t \alpha(\comp^\theta \hol(v) )
\,
\overline{\chi_p(t \alpha(\comp^\theta \hol(w) )}
\intd \theta \intd t
\intd \haar(\omega)
=
\bigo \left( \frac{1}{R^\gamma} \right)
\]
for some $0 < \gamma < 1$ by Lemma~\ref{lem:flowLengthEstimate}.
For each $\omega \in K$ consider the inner sum over $w$ above in two parts, according to whether $w$ satisfies \eqref{eqn:radialFriends}.
By Lemma~\ref{lem:unfriendlyCount} and the trivial bound the sum over those $w$ not satisfying \eqref{eqn:radialFriends} is $\bigo(R^{1+\rho})$ which is good enough because $\rho < 1$.
For $w$ that do satisfy \eqref{eqn:radialFriends} we apply Lemma~\ref{lem:radialFriends}.
In combination we obtain the estimate
\begin{align*}
&
\left|
\frac{1}{R^4} \sum_{v \in \Lambda(\alpha;R)} \sum_{w \in \Lambda(\alpha;R)}
\frac{1}{2 \pi} \int\limits_0^{2\pi}
\frac{1}{T} \int\limits_0^T
\chi_p(\nbar \comp^\theta v \nbar[2] + t \alpha(\comp^\theta v))
\chi_p(\nbar \comp^\theta w \nbar[2] + t \alpha(\comp^\theta w))
\intd t \intd \theta
\right|
\\
\le
{}
&
\frac{1}{R^4} (c_4 R^2) (c_3 R^{1+\rho})
+
\frac{1}{R^4} (c_4 R^2)^2
\left(
4 \cdot \frac{1 + \epsilon}{2 \pi R}
+
\frac{8}{\pi^2}
\sqrt{\frac{6|p|}{\epsilon}} \left( \log R + \log \frac{\pi}{4} \right) \frac{R^{\frac{1}{2}}}{R^\rho}
\right)
\end{align*}
which gives the theorem because $1 > \rho > \frac{1}{2}$.
\end{proof}

\subsection{Proof of Theorem~\ref{thm:mainTheorem}}
\label{subsec:finale}

Here we combine the preceding subsections to prove Theorem~\ref{thm:mainTheorem}.

\begin{proof}[Proof of Theorem~\ref{thm:mainTheorem}]
Fix a sequence $\tau_1 > \tau_2 > \cdots \to 1$.
The compact sets $K$ exhause almost all of $\mathcal{H}$ as $\Xi \to \infty$ and $r \to 0$.
It therefore suffices by Lemma~\ref{lem:tauberian} and Lemma~\ref{lem:orderToDiscs} to prove for every $i \in \N$ that \eqref{eqn:orderToDiscs} holds for $\haar$ almost every $\omega \in K$.
Fix $i \in \N$ and write $\tau$ for $\tau_i$.
Let $\gamma$ be as in Theorem~\ref{thm:mainEstimate}.
Taking $\sigma$ small enough and $R^2 = \ceil{\tau^J}/c$ in Theorem~\ref{thm:orbitSmear} and applying using both Corollary~\ref{cor:ellipseEstimate} and Theorem~\ref{thm:mainEstimate} gives $\psi > 0$ such that
\[
\haar \left( \left\{ \omega \in K : \left| \frac{c}{\ceil{\tau^J}} \sum_{v \in \Lambda(\omega;\sqrt{\ceil{\tau^J}/c})} \chi_p(\nbar \hol(v) \nbar) \right|^2 \ge \left( \frac{c}{\ceil{\tau^J}} \right)^{\sigma/2} \right\} \right)
=
\bigo \left( \frac{1}{\ceil{\tau^J}^\psi} \right)
\]
for all $J \in \N$ large enough.
Since the right-hand side is summable we conclude from the Borel-Cantelli lemma that
\[
\lim_{J \to \infty} \frac{c}{\ceil{\tau^J}} \sum_{v \in \Lambda(\omega;\sqrt{\ceil{\tau^J}/c})} \chi_p(\nbar \hol(v) \nbar) = 0
\]
as desired.
\end{proof}

\appendix
\renewcommand{\thesection}{\Alph{section}}

\section{Equidistribution of the lengths of the primitive vectors in integer lattices by Daniel El-Baz and Bingrong Huang}

For a real number $x$, let $\lfloor x \rfloor$ be the largest integer less than or equal to $x$, and let $\{x\}=x-\lfloor x\rfloor$ be its fractional part. Let $\|\cdot\|$ be the usual Euclidean norm in $\mathbb{R}^2$.
Let $\mathbb{Z}_{\mathrm{prim}}^2=\{ (a,b)\in \mathbb{Z}^2: \gcd(a,b)=1 \}$ and $\mathcal{S}_R=\{\mathbf{v} \in \mathbb{Z}_{\mathrm{prim}}^2 \, : \, \|g \mathbf{v}\| \le R\}$.

\begin{theorem}
For every $g\in \mathrm{GL}_2(\mathbb{R})$, the sequence  $(\{\|g \mathbf{v}\|\})_{\mathbf{v}\in \mathcal S_R}$ 
is uniformly distributed as $R\rightarrow \infty$.
\end{theorem}

\begin{proof}
  We want to estimate, for every $\alpha \in (0,1)$, the following quantity:
  \begin{equation*}
    N_R(\alpha) = \#\{ \mathbf{v} \in \mathcal{S}_R \, : \, \{\|g \mathbf{v}\|\} \le \alpha \}.
  \end{equation*}
  We rewrite
  \begin{align*}
    N_R(\alpha) &= \sum_{d \le \lfloor R \rfloor} \sum_{\substack{d \le \|g \mathbf{v}\| \le d + \alpha \\ \|g \mathbf{v}\| \le R}} 1.
  \end{align*}
  Let $\mathcal D_g = \{ \mathbf{x}\in\mathbb{R}^2 : \|g\mathbf{x}\| \leq 1 \}$. 
  By the prime number theorem estimate
  for the primitive lattice point problem for a compact convex domain with smooth boundary (which our elliptic domain $\mathcal{D}_g$ certainly is) \cite[Equation 1.6]{MR1397317}, we have
  \[
    P(y) := \#\{  \mathbf{v}\in \mathbb{Z}_{\mathrm{prim}}^2 \, : \, \|g \mathbf{v}\| \le y\} = \sum_{\substack{\mathbf{v}\in \mathbb{Z}_{\mathrm{prim}}^2 \\ \|g \mathbf{v}\|\leq y}} 1 = \frac{6}{\pi^2} \mathrm{area}(\mathcal D_g) y^2 + o(y).
  \]
  Hence we obtain
  \begin{align*}
    N_R(\alpha) 
    & = \sum_{d \le \lfloor R \rfloor - 1} ( P(d+\alpha) -P(d) )
    + O(R)  \\
    & = \sum_{d \le \lfloor R \rfloor - 1} \left(\frac 6{\pi^2} \mathrm{area}(\mathcal D_g) 2d\alpha + o(d) \right)
    + O(R) \\
    & = \frac 6{\pi^2} \mathrm{area}(\mathcal D_g) R^2 \alpha + o(R^2).
  \end{align*}
  This implies our claim.
\end{proof}

\printbibliography

\end{document}